\documentclass[12pt,leqno]{amsart}
\usepackage[latin9]{inputenc}
\usepackage{verbatim}
\usepackage{amsthm}
\usepackage{amstext}
\usepackage{amssymb}
\usepackage[dvipdfmx]{graphicx}
\usepackage{xcolor}
\usepackage{soul}

\usepackage{mathrsfs}
\usepackage{amscd}

\makeatletter
\numberwithin{equation}{section}

\usepackage{amsthm}\usepackage{amscd}\usepackage{bm}

\usepackage[dvipdfmx,pdfstartview=FitH,pdfborderstyle={/S/B/W 1}]{hyperref}

\theoremstyle{plain}
\newtheorem{main theorem}{Main Theorem}
\newtheorem{theorem}{Theorem}[section]
\newtheorem{lemma}[theorem]{Lemma}

\newtheorem{proposition}[theorem]{Proposition}

\theoremstyle{definition}

\newtheorem{remark}[theorem]{Remark}

\newcommand{\mdim}{\mathrm{mdim}}
\newcommand{\diam}{\mathrm{diam}}

\newcommand{\rdim}{\mathrm{rdim}}

\makeatother

\begin{document}

\title[Symbolic dynamics in mean dimension theory]{Symbolic dynamics in mean dimension theory}

\author{Mao Shinoda, Masaki Tsukamoto}



\date{\today}

\thanks{M.S. was partially supported by Grant-in-Aid for JSPS Research Fellow, JSPS KAKENHI Grant Number 17J03495.
M.T. was partially supported by JSPS KAKENHI 18K03275. }

\maketitle

\begin{abstract}
Furstenberg (1967) calculated the Hausdorff and Minkowski dimensions of one-sided subshifts
in terms of topological entropy.
We generalize this to $\mathbb{Z}^2$-subshifts.
Our generalization involves mean dimension theory.
We calculate the metric mean dimension and mean Hausdorff dimension of $\mathbb{Z}^2$-subshifts 
with respect to a subaction of $\mathbb{Z}$.
The resulting formula is quite analogous to Furstenberg's theorem.
We also calculate the rate distortion dimension of $\mathbb{Z}^2$-subshifts in terms of Kolmogorov--Sinai entropy.
\end{abstract}

\section{Introduction}  \label{section: introduction}

\subsection{Hausdorff and Minkowski dimensions of subshifts}  \label{subsection: Hausdorff and Minkowski dimensions of subshifts}

Let $A$ be a finite set (alphabet).
We consider the one-sided infinite product $A^\mathbb{N} = A\times A\times A\times \cdots$
with the shift map $\sigma: A^\mathbb{N}\to A^\mathbb{N}$ defined by 
\[  \sigma\left((x_n)_{n\in \mathbb{N}}\right) = (x_{n+1})_{n\in \mathbb{N}}. \]
Take $\alpha>1$. We define a distance $d$ on $A^\mathbb{N}$ by 
\[   d(x,y) = \alpha^{-\min\{n \, |\, x_n\neq y_n\}}.    \]
Let $\mathcal{X}\subset A^\mathbb{N}$ be a $\sigma$-invariant closed subset.
Furstenberg \cite[Proposition III.1]{Furstenberg} calculated the Hausdorff and Minkowski dimensions of $\mathcal{X}$ with 
respect to $d$:
\begin{equation} \label{eq: Furstenberg theorem}
     \dim_{\mathrm{H}}(\mathcal{X},d) = \dim_{\mathrm{M}}(\mathcal{X}, d) = \frac{h_{\mathrm{top}}(\mathcal{X},\sigma)}{\log \alpha}. 
\end{equation}     
Here $h_{\mathrm{top}}(\mathcal{X}, \sigma)$ is the topological entropy of $(\mathcal{X}, \sigma)$.
The purpose of the paper is to extend this result to \textit{higher rank actions}.

\subsection{Mean dimension theory}  \label{subsection: mean dimension theory}

Mean dimension theory provides a meaningful framework for extending (\ref{eq: Furstenberg theorem}) to 
higher rank actions.
This is the theory first introduced by Gromov \cite{Gromov} and further developed 
by Lindenstrauss--Weiss \cite{Lindenstrauss--Weiss}, Lindenstrauss \cite{Lindenstrauss}, and more recently 
Lindenstrauss and the second named author \cite{Lindenstrauss--Tsukamoto double VP}.
We review the basic ingredients here. (The precise definitions will be given in \S \ref{section: preliminaries}.)

A pair $(\mathcal{X}, T)$ is called a \textbf{dynamical system} if $\mathcal{X}$ is a compact metric space 
and $T:\mathcal{X}\to \mathcal{X}$ is a homeomorphism\footnote{We can also consider a non-invertible map $T$ as in 
\S \ref{subsection: Hausdorff and Minkowski dimensions of subshifts}. But we consider only invertible $T$ here for simplicity.}.
Gromov \cite{Gromov} defined {\bf mean topological dimension} $\mdim(\mathcal{X},T)$.
This is a dynamical analogue of topological dimension, and it evaluates the number of 
parameters per iterate for describing the orbits of $(\mathcal{X}, T)$.
As the name suggested, the mean topological dimension is a \textit{topological invariant} of dynamical systems.
There are many important works around this quantity
\cite{Lindenstrauss--Weiss, Lindenstrauss, Gutman Jaworski theorem, 
Gutman--Lindenstrauss--Tsukamoto, Gutman--Tsukamoto minimal, Li--Liang, Tsukamoto Brody, Meyerovitch--Tsukamoto,
Lindenstrauss--Tsukamoto double VP}.
However mean topological dimension is \textit{not} the right notion for the purpose of this paper because 
Furstenberg's theorem (\ref{eq: Furstenberg theorem}) concerns with Hausdorff and Minkowski dimensions, not topological one.
(The topological dimension of a subshift $\mathcal{X}\subset A^\mathbb{N}$ is simply zero.)

Let $d$ be a metric (i.e. a distance function) on $\mathcal{X}$. 
Lindenstrauss--Weiss \cite{Lindenstrauss--Weiss} defined {\bf metric mean dimension} $\mdim_{\mathrm{M}}(\mathcal{X}, T, d)$.
This is a dynamical analogue of Minkowski dimension.
Lindenstrauss and the second named author \cite{Lindenstrauss--Tsukamoto double VP} 
defined {\bf mean Hausdorff dimension} $\mdim_{\mathrm{H}}(\mathcal{X},T,d)$.
This is a dynamical analogue of Hausdorff dimension.
Metric mean dimension and mean Hausdorff dimension are metric dependent quantities.
They provide a good framework for the purpose of the paper.

It is well-known in geometric measure theory \cite{Mattila} that 
metrical dimensions are deeply connected to measure theory.
In particular we can introduce the concept of (metric dependent) dimension for measure 
(see e.g. \cite{Renyi, Young, Kawabata--Dembo}).
Similarly we can introduce a mean dimensional quantity for invariant measures of dynamical systems.
Let $\mu$ be a $T$-invariant Borel probability measure on $\mathcal{X}$.
Let $X$ be a random variable taking values in $\mathcal{X}$ according to the law $\mu$, and 
we consider the stochastic process $\{T^n X\}_{n\in \mathbb{Z}}$. 
We denote by $R(d,\mu,\varepsilon)$ $(\varepsilon>0)$ the \textit{rate distortion function} of this stochastic process.
This is the key quantity of Shannon's rate distortion theory \cite{Shannon, Shannon59}.
It evaluates how many bits per iterate we need for describing the process within the distortion (with respect to $d$) bound by 
$\varepsilon$.
Following Kawabata--Dembo \cite{Kawabata--Dembo}, we define the 
{\bf upper and lower rate distortion dimensions} by\footnote{Throughout the paper, we assume that the base of the logarithm is two.}
  \begin{equation} \label{eq: rate distortion dimension}
      \overline{\rdim}(\mathcal{X}, T, d, \mu) = \limsup_{\varepsilon\to 0} \frac{R(d,\mu,\varepsilon)}{\log(1/\varepsilon)}, 
      \quad 
      \underline{\rdim}(\mathcal{X},T,d,\mu) = \liminf_{\varepsilon \to 0} \frac{R(d,\mu,\varepsilon)}{\log(1/\varepsilon)}.         
    \end{equation}
    When the upper and lower limits coincide, we denote the common value by $\rdim(\mathcal{X},T,d,\mu)$.
    
Metric mean dimension, mean Hausdorff dimension and rate distortion dimension are related to each other.
See Proposition \ref{proposition: metric mean dimension bounds mean Hausdorff dimension} and 
Theorem \ref{theorem: rate distortion dimension and mean dimensions} below.

\subsection{Statement of the main result}  \label{subsection: statement of the main result}

Let $A$ be a finite set as in \S \ref{subsection: Hausdorff and Minkowski dimensions of subshifts}.
We consider the infinite product $A^{\mathbb{Z}^2}$ index by $\mathbb{Z}^2$.
We define the shifts $\sigma_1$ and $\sigma_2$ on $A^{\mathbb{Z}^2}$ by 
\[  \sigma_1\left((x_{m,n})_{m,n\in \mathbb{Z}}\right) = (x_{m+1,n})_{m,n\in \mathbb{Z}}, \quad 
      \sigma_2\left((x_{m,n})_{m,n\in \mathbb{Z}}\right) = (x_{m,n+1})_{m,n\in \mathbb{Z}}.  \]
Fix $\alpha>1$ and define a distance $d$ on $A^{\mathbb{Z}^2}$ by 
\begin{equation} \label{eq: metric on shift}
     d(x,y) = \alpha^{-\min \{ |u|_\infty |\, x_u \neq y_u \}}, 
\end{equation}     
where $|u|_\infty = \max(|m|, |n|)$ for $u = (m,n)\in \mathbb{Z}^2$.
We call a closed subset $\mathcal{X}\subset A^{\mathbb{Z}^2}$ \textbf{subshift} if it is invariant under both $\sigma_1$ and $\sigma_2$.

The following is our main result.

\begin{theorem}  \label{theorem: higher rank analogue of Furstenberg}
Let $\mathcal{X}\subset A^{\mathbb{Z}^2}$ be a subshift. Then 
\begin{equation} \label{eq: mean Hausdorff dimension of subshifts}
     \mdim_{\mathrm{H}} (\mathcal{X}, \sigma_1, d) = \mdim_{\mathrm{M}} (\mathcal{X},\sigma_1,d) = 
     \frac{2 h_{\mathrm{top}}(\mathcal{X},\sigma_1,\sigma_2)}{\log \alpha}. 
\end{equation}     
Here $h_{\mathrm{top}}(\mathcal{X},\sigma_1,\sigma_2)$ is the topological entropy of $(\mathcal{X}, \sigma_1,\sigma_2)$.
Moreover, if $\mu$ is a Borel probability measure on $\mathcal{X}$ invariant under both $\sigma_1$ and $\sigma_2$ then 
\begin{equation*}  
  \rdim(\mathcal{X},\sigma_1,d, \mu) = \frac{2 h_\mu(\mathcal{X},\sigma_1,\sigma_2)}{\log \alpha}. 
\end{equation*}  
Here $h_\mu(\mathcal{X},\sigma_1,\sigma_2)$ is the Kolmogorov--Sinai entropy of $(\mathcal{X}, \sigma_1,\sigma_2)$
with respect to the measure $\mu$.
\end{theorem}

In particular if $\mu$ is a maximal entropy measure 
(i.e. $ h_\mu(\mathcal{X},\sigma_1,\sigma_2) =  h_{\mathrm{top}}(\mathcal{X},\sigma_1,\sigma_2)$) then 
$\rdim(\mathcal{X},\sigma_1,d,\mu)$ coincides with the mean Hausdorff dimension and metric mean dimension.

The point of the statement is that we consider various mean dimensional quantities for the action of $\sigma_1$, 
not the total $\mathbb{Z}^2$-action generated by $\sigma_1$ and $\sigma_2$.
In other words we consider only $\sigma_1$ and disregard $\sigma_2$.
Nevertheless we can recover the entropy of the total $\mathbb{Z}^2$-action.
This might look a bit strange at first sight. But in fact it has the same spirit with Furstenberg's theorem (\ref{eq: Furstenberg theorem}).
In (\ref{eq: Furstenberg theorem}), we consider the Hausdorff and Minkowski dimensions of one-sided subshifts.
Hausdorff and Minkowski dimensions are purely metric invariants and do not involve dynamics.
So here we disregard the action at all.
However we can recover the topological entropy. 
See Remark \ref{remark on main theorem} (4) below for more backgrounds behind the formulation of the theorem.

\begin{remark} \label{remark on main theorem}
\begin{enumerate}
   \item Subshifts $\mathcal{X}\subset A^{\mathbb{Z}^2}$ are totally disconnected.
So the mean topological dimension of $(\mathcal{X}, \sigma_1)$ is zero.

    \item Probably some readers notice a slight difference between Furstenberg's theorem (\ref{eq: Furstenberg theorem})
    and our (\ref{eq: mean Hausdorff dimension of subshifts}):
   Our formula involves the coefficient ``$2$'' wheres Furstenberg's theorem does not.
   This difference comes from the point that Furstenberg's theorem considers one-sided shifts (i.e. actions of $\mathbb{N}$, not $\mathbb{Z}$).
   If we consider two-sided shifts, then we get a result completely analogous to (\ref{eq: mean Hausdorff dimension of subshifts}).
    
   \item  Theorem \ref{theorem: higher rank analogue of Furstenberg} 
   can be generalized to $\mathbb{Z}^k$-shifts and, probably some noncommutative group actions. 
   But we stick to $\mathbb{Z}^2$ for simplicity of the exposition.
   
    \item A guiding principle behind our theorem is as follows: 
    Let $T:\mathbb{Z}^k\times \mathcal{X}\to \mathcal{X}$
    be a continuous action of $\mathbb{Z}^k$ on a compact metric space $\mathcal{X}$.
    If $T$ has some hyperbolicity-like property, then we can control the mean dimensional quantities of 
   the restriction of $T$ to subgroups $G \subset \mathbb{Z}^k$ with $\mathrm{rank}\, G = k-1$. 
  
    When $k=1$, the subgroup $G$ must be trivial. So, in particular, this principle claims that we can control the dimensions of 
  $\mathcal{X}$ if $\mathcal{X}$ admits an action of $\mathbb{Z}$ with some ``hyperbolicty''.
  Furstenberg's theorem (\ref{eq: Furstenberg theorem}) is a typical example of such results because 
  symbolic dynamics can be seen as an extreme case of hyperbolic dynamics.
  Theorem (\ref{theorem: higher rank analogue of Furstenberg}) corresponds to the case of $k=2$ in this principle.
  
  Another manifestation of the above principle was given by the work of \cite{Meyerovitch--Tsukamoto}.
  They proved that if $T:\mathbb{Z}^k\times \mathcal{X}\to \mathcal{X}$ is expansive then 
  the mean topological dimension of $T|_G$ is finite for any rank $(k-1)$ subgroups $G\subset \mathbb{Z}^k$.
  In particular, when $k=1$, a compact metric space has finite topological dimension if 
  it admits an expansive action of $\mathbb{Z}$.
  This is a classical theorem of Ma\~{n}\'{e} \cite{Mane}.  
  
  \item We consider the action of $\sigma_1$ in Theorem \ref{theorem: higher rank analogue of Furstenberg}.
  This corresponds to a study of the action of the subgroup $\{(n, 0)|\, n\in \mathbb{Z}\} \subset \mathbb{Z}^2$.
  According to the principle in the above (4), it is also natural to consider other rank-one subgroups.
  Namely we should study various mean dimensional quantities 
  for the action of $\sigma_1^a \sigma_2^b$ for any nonzero $(a,b)\in \mathbb{Z}^2$, which 
  corresponds to the subgroup $\{(an, bn)|\, n\in \mathbb{Z}\}$.
 
   Indeed we can calculate them.
  Take a nonzero $(a,b)\in \mathbb{Z}^2$.
   Then for a subshift $\mathcal{X}\subset A^{\mathbb{Z}^2}$ we have 
  \begin{equation} \label{eq: mean Hausdorff dimension for sigma_1^a sigma_2^b}
     \begin{split}
      \mdim_{\mathrm{M}}(\mathcal{X}, \sigma_1^a \sigma_2^b, d) &= 
        \mdim_{\mathrm{H}}(\mathcal{X}, \sigma_1^a \sigma_2^b, d) = 
         2(|a|+|b|) \frac{h_{\mathrm{top}}(\mathcal{X},\sigma_1, \sigma_2)}{\log \alpha}, \\
      \rdim(\mathcal{X}, \sigma_1^a \sigma_2^b, d, \mu) &= 2(|a|+|b|) \frac{h_\mu(\mathcal{X}, \sigma_1,\sigma_2)}{\log\alpha}.
     \end{split}       
  \end{equation}       
  Here $d$ is the metric defined by (\ref{eq: metric on shift}) and 
  $\mu$ is a Borel probability measure on $\mathcal{X}$ invariant under $\sigma_1$ and $\sigma_2$.

  The factor $2(|a|+|b|)$ in (\ref{eq: mean Hausdorff dimension for sigma_1^a sigma_2^b}) has the following geometric meaning.  
  For natural numbers $M$ and $N$, we define $\Lambda_{a,b}(M,N) \subset \mathbb{Z}^2$ as the set of points
  \[  (an + x, bn +y), \quad (0\leq n<N, \>  |(x,y)|_\infty <M). \]  
  Here $n, x, y$ are integers. 
  (Namely, we consider the parallel translations of $(-M,M)^2$
  along the segment $\{(an, bn)|\, 0\leq n < N\}$. )
  Then we have
  \[ 2(|a|+|b|) =  \lim_{M\to \infty} \left(\lim_{N\to \infty} \frac{|\Lambda_{a,b}(M,N)|}{MN}\right) \quad 
      (\text{$|\cdot|$ denotes the cardinality}).    \]

The square $(-M, M)^2$ is the disk of radius $M$ in the $\ell^\infty$-norm $|u|_\infty$. 
The relevance of the $\ell^\infty$-norm here comes from the point that the metric (\ref{eq: metric on shift})
uses it. If we use a different metric, then we get a different result.
For example, consider the following metric $\rho$ on $A^{\mathbb{Z}^2}$:
 \begin{equation}  \label{eq: ell^2 metric on shift}
     \rho(x,y) = \alpha^{-\min\{ \sqrt{m^2+n^2} |\, x_{m,n} \neq y_{m,n}\}}. 
 \end{equation}    
 This metric uses the $\ell^2$-norm $\sqrt{m^2+n^2}$ instead of the $\ell^\infty$-norm.
 For this metric we have 
   \begin{equation} \label{eq: mean Hausdorff dimension for sigma_1^a sigma_2^b for ell^2 metric}
     \begin{split}
      \mdim_{\mathrm{M}}(\mathcal{X}, \sigma_1^a \sigma_2^b, \rho) &= 
        \mdim_{\mathrm{H}}(\mathcal{X}, \sigma_1^a \sigma_2^b, \rho) = 
         2\sqrt{a^2+b^2} \cdot \frac{h_{\mathrm{top}}(\mathcal{X},\sigma_1, \sigma_2)}{\log \alpha}, \\
      \rdim(\mathcal{X}, \sigma_1^a \sigma_2^b,\rho, \mu) &= 2\sqrt{a^2+b^2}\cdot  
      \frac{h_\mu(\mathcal{X}, \sigma_1,\sigma_2)}{\log\alpha}.
     \end{split}       
  \end{equation}       
  The proofs of (\ref{eq: mean Hausdorff dimension for sigma_1^a sigma_2^b}) and 
  (\ref{eq: mean Hausdorff dimension for sigma_1^a sigma_2^b for ell^2 metric}) are conceptually the same with 
 the proof of Theorem \ref{theorem: higher rank analogue of Furstenberg}.
 However they become notationally more messy.
 So we decide to concentrate on the statement of Theorem \ref{theorem: higher rank analogue of Furstenberg}.
 It clarifies the ideas in the simplest form.
\end{enumerate}    
\end{remark}

\subsection*{Acknowledgment}
The first proof we gave to Theorem \ref{theorem: higher rank analogue of Furstenberg}
contained a gap. Elon Lindenstrauss pointed out this, and he also kindly explained to us how to fix the gap.
We would like to thank him for the help.

\section{Preliminaries}  \label{section: preliminaries}

The purpose of this section is to define the three dynamical dimensions 
(metric mean dimension, mean Hausdorff dimension and rate distortion dimension)\footnote{We do not use mean topological dimension 
in the paper. So we skip to define it.}
and explain some of their basic properties.

\subsection{Metric mean dimension and mean Hausdorff dimension}  
\label{subsection: metric mean dimension and mean Hausdorff dimension}

Let $(\mathcal{X},d)$ be a compact metric space.
For $\varepsilon>0$
we define $\#(\mathcal{X}, d, \varepsilon)$ as the minimum natural number $n$ such that 
$\mathcal{X}$ can be covered by open sets $U_1, \dots, U_n$ with $\diam \, U_i < \varepsilon$ for all $1\leq i\leq n$.
The upper and lower Minkowski dimensions of $(\mathcal{X},d)$ are given by
\[  \overline{\dim}_{\mathrm{M}}(\mathcal{X},d) = \limsup_{\varepsilon \to 0} \frac{\log\#(\mathcal{X},d,\varepsilon)}{\log(1/\varepsilon)},
  \quad
   \underline{\dim}_{\mathrm{M}}(\mathcal{X},d) = \liminf_{\varepsilon\to 0} \frac{\log\#(\mathcal{X},d,\varepsilon)}{\log(1/\varepsilon)}. \] 
For $s\geq 0$ and $\varepsilon >0$ we define 
\[ \mathcal{H}_\varepsilon^s(\mathcal{X},d) = \inf\left\{ \sum_{i=1}^\infty (\diam \, E_i)^s \middle| \,
    \mathcal{X} = \bigcup_{n=1}^\infty E_i \text{ with $\diam\, E_i < \varepsilon$ for all $i\geq 1$}\right\}. \]
Here we use the convention that $0^0 = 1$ and $\diam(\emptyset)^s = 0$.    
Since $\mathcal{X}$ is compact, this is equal to the infimum of 
\[ \sum_{i=1}^n (\diam\, U_i)^s \]
over all finite open covers $\{U_1,\dots, U_n\}$ of $\mathcal{X}$ with $\diam\, U_i < \varepsilon$ for all $1\leq i\leq n$.
We set 
\[  \dim_{\mathrm{H}}(\mathcal{X},d,\varepsilon) = \sup\{s\geq 0|\, \mathcal{H}^s_\varepsilon(\mathcal{X},d) \geq 1\}. \]
The Hausdorff dimension of $(\mathcal{X},d)$ is given by 
\[ \dim_{\mathrm{H}}(\mathcal{X},d) = \lim_{\varepsilon \to 0} \dim_{\mathrm{H}}(\mathcal{X},d,\varepsilon).  \]

Given a homeomorphism $T:\mathcal{X}\to \mathcal{X}$, we define metrics $d^T_N$ $(N\geq 1)$ on $\mathcal{X}$ by 
\[  d^T_N(x,y) = \max_{0\leq n < N} d(T^n x, T^n y). \]
We define the \textbf{entropy at the resolution $\varepsilon >0$} by 
\[ S(\mathcal{X}, T, d, \varepsilon) = \lim_{N\to \infty} \frac{\log\#(\mathcal{X},d^T_N,\varepsilon)}{N}. \] 
This limit exists because $\log\#(\mathcal{X},d^T_N,\varepsilon)$ is suadditive in $N$.
We define the \textbf{upper and lower metric mean dimensions} by 
\[  \overline{\mdim}_{\mathrm{M}}(\mathcal{X},T,d) = 
    \limsup_{\varepsilon\to 0} \frac{S(\mathcal{X},T,d,\varepsilon)}{\log(1/\varepsilon)}, \quad 
   \underline{\mdim}_{\mathrm{M}}(\mathcal{X},T,d) = 
    \liminf_{\varepsilon\to 0} \frac{S(\mathcal{X},T,d,\varepsilon)}{\log(1/\varepsilon)}.  \]
When the upper and lower limits coincide, we denote the common value by $\mdim_{\mathrm{M}}(\mathcal{X}, T,d)$.

We define the \textbf{upper and lower mean Hausdorff dimensions} by 
\begin{equation*}
  \begin{split}
  \overline{\mdim}_{\mathrm{H}}(\mathcal{X},T,d) &= \lim_{\varepsilon\to 0} 
    \left(\limsup_{N\to \infty} \frac{\dim_{\mathrm{H}}(\mathcal{X},d_N,\varepsilon)}{N}\right), \\
    \underline{\mdim}_{\mathrm{H}}(\mathcal{X},T,d)& = \lim_{\varepsilon\to 0} 
    \left(\liminf_{N\to \infty} \frac{\dim_{\mathrm{H}}(\mathcal{X},d_N,\varepsilon)}{N}\right). 
  \end{split}  
\end{equation*}    
When these two quantities are equal to each other, we denote the common value by 
$\mdim_{\mathrm{H}}(\mathcal{X},T,d)$.

The following is the dynamical analogue of the fact that Minkowski dimension bounds Hausdorff dimension.
It was proved in \cite[Proposition 3.2]{Lindenstrauss--Tsukamoto double VP}.

\begin{proposition}  \label{proposition: metric mean dimension bounds mean Hausdorff dimension}
\[   \underline{\mdim}_{\mathrm{H}}(\mathcal{X},T,d) \leq   \overline{\mdim}_{\mathrm{H}}(\mathcal{X},T,d)
      \leq    \underline{\mdim}_{\mathrm{M}}(\mathcal{X},T,d)  \leq  \overline{\mdim}_{\mathrm{M}}(\mathcal{X},T,d). \]
\end{proposition}

\begin{remark}
Here is one remark about the notation.
In the paper \cite{Lindenstrauss--Tsukamoto double VP}, the lower mean Hausdorff dimension played no role.
So the upper mean Hausdorff dimension was simply denoted by $\mdim_{\mathrm{H}}(\mathcal{X},T,d)$
in \cite{Lindenstrauss--Tsukamoto double VP}.   
\end{remark}

\subsection{Mutual information} \label{subsection: mutual information}

Let $(\Omega ,\mathbb{P})$ be a probability space.
Let $\mathcal{X}$ and $\mathcal{Y}$ be measurable spaces, and let 
$X:\Omega\to \mathcal{X}$ and $Y:\Omega\to \mathcal{Y}$ be measurable maps.
We want to define their \textbf{mutual information} $I(X;Y)$ as the measure of the amount of information 
$X$ and $Y$ share. (This will be used in the definition of rate distortion function in the next subsection.)
The basic reference is \cite{Cover--Thomas}.

\textbf{Case 1: When $\mathcal{X}$ and $\mathcal{Y}$ are finite sets.}
In this case\footnote{We always assume that the $\sigma$-algebras of finite sets are the largest ones, i.e. the sets of 
all subsets.} we set 
\begin{equation} \label{eq: mutual information}
  \begin{split}
    I(X;Y)  &= H(X)+H(Y)-H(X,Y) = H(X)- H(Y|X) \\
             &= \sum_{x\in \mathcal{X}, y\in \mathcal{Y}} \mathbb{P}(X=x,Y=y) 
           \log \frac{\mathbb{P}(X=x, Y=y)}{\mathbb{P}(X=x) \mathbb{P}(Y=y)}.
   \end{split} 
\end{equation}    
Here we have used the convention that $0\log(0/a)= 0$ for all $a\geq 0$.

\textbf{Case 2: General case.}
Let $f:\mathcal{X}\to A$ and $g:\mathcal{Y}\to B$ be measurable maps such that $A$ and $B$ are finite sets.
Then we can define $I(f\circ X; g\circ Y)$ by (\ref{eq: mutual information}).
We define $I(X;Y)$ as the supremum of $I(f\circ X; g\circ Y)$ over all finite range measurable maps $f$ on $\mathcal{X}$ and
$g$ on $\mathcal{Y}$.
When $\mathcal{X}$ and $\mathcal{Y}$ are finite sets, this definition is compatible with (\ref{eq: mutual information}).
(Namely the supremum is attained when $f$ and $g$ are the identity maps.)

The mutual information is symmetric and nonnegative: $I(X;Y) = I(Y;X) \geq 0$.
The following basic result immediately follows from the above definition.

\begin{lemma}[Data-Processing inequality]  \label{lemma: data-processing inequality}
Let $\mathcal{Z}$ and $\mathcal{W}$ be measurable spaces.
If $f:\mathcal{X}\to \mathcal{Z}$ and $g:\mathcal{Y}\to \mathcal{W}$ be measurable maps, then $I(f(X); g(Y)) \leq I(X;Y)$.
\end{lemma}

\subsection{Rate distortion theory}  \label{subsection: rate distortion theory}

Here we introduce rate distortion function.
As Shannon entropy is the fundamental limit of \textit{lossless} data compression, 
rate distortion function is the fundamental limit of \textit{lossy} data compression\footnote{For example, expanding a given signal 
in a wavelet basis and discarding a small terms.}.
A friendly introduction can be found in \cite[Chapter 10]{Cover--Thomas}.

Let $(\mathcal{X},T)$ be a dynamical system with a metric $d$ and an invariant Borel probability measure $\mu$.
We define the \textbf{rate distortion function} $R(d,\mu,\varepsilon)$ $(\varepsilon>0)$ as the infimum of 
\[ \frac{I(X;Y)}{N}, \]
where $N$ runs over natural numbers, $X$ and $Y=(Y_0,\dots, Y_{N-1})$ are random variables defined on some probability space
$(\Omega, \mathbb{P})$ such that 
\begin{itemize}
   \item $X$ takes values in $\mathcal{X}$ according to the law $\mu$. 
   \item $Y_0, \dots, Y_{N-1}$ take values in $\mathcal{X}$ and satisfy 
   \begin{equation} \label{eq: distortion condition}
        \mathbb{E}\left(\frac{1}{N}\sum_{n=0}^{N-1} d(T^n X, Y_n)\right) < \varepsilon. 
   \end{equation}     
\end{itemize}

The condition (\ref{eq: distortion condition}) means that $Y=(Y_0, \dots, Y_{N-1})$ approximates the stochastic process $X, TX, \dots, T^{N-1}X$
within the averaged distortion bound by $\varepsilon$.
We define the upper and lower rate distortion dimensions $\overline{\rdim}(\mathcal{X},T,d, \mu)$ and 
$\underline{\rdim}(\mathcal{X},T,d,\mu)$ by (\ref{eq: rate distortion dimension})
in \S \ref{subsection: mean dimension theory}.

The rate distortion function $R(d,\mu,\varepsilon)$ is the minimum rate when we try to quantize the process
$\{T^n X\}_{n\in \mathbb{Z}}$ within the averaged distortion bound by $\varepsilon$.
See \cite[Chapter 10]{Cover--Thomas}, \cite[Chapter 11]{Gray} and \cite{ECG, LDN} for the precise meaning of this statement.

The rest of this subsection is not used in the proof of Theorem \ref{theorem: higher rank analogue of Furstenberg}.
We include this for providing readers a wider view of the subject.
A metric $d$ is said to have the \textbf{tame growth of covering numbers} if for any $\delta>0$
\begin{equation} \label{eq: tame growth of covering numbers}
   \lim_{\varepsilon\to 0} \varepsilon^\delta \log \#(\mathcal{X}, d, \varepsilon) =0. 
\end{equation}    
Note that this is purely a condition on the metric structure and does not involve dynamics.
(\ref{eq: tame growth of covering numbers}) is a mild condition. 
It is known (\cite[Lemma 3.10]{Lindenstrauss--Tsukamoto double VP})
that every compact metrizable space admits a metric satisfying (\ref{eq: tame growth of covering numbers}).
For example, the metrics (\ref{eq: metric on shift}) and (\ref{eq: ell^2 metric on shift}) on the shift space $A^{\mathbb{Z}^2}$ satisfy 
(\ref{eq: tame growth of covering numbers}).
The following theorem \cite[Proposition 3.2, Theorem 3.11]{Lindenstrauss--Tsukamoto double VP} 
provides a link between rate distortion dimension and various mean dimensions.
Here we denote by $\mathscr{M}^T(\mathcal{X})$ the set of all invariant Borel probability measures on $\mathcal{X}$.

\begin{theorem} \label{theorem: rate distortion dimension and mean dimensions}
\[  \overline{\rdim}(\mathcal{X},T,d,\mu) \leq \overline{\mdim}_{\mathrm{M}}(\mathcal{X},T,d), \quad 
    \underline{\rdim}(\mathcal{X},T,d,\mu) \leq \underline{\mdim}_{\mathrm{M}}(\mathcal{X},T,d). \]
If $d$ has the tame growth of covering numbers then 
\[  \overline{\mdim}_{\mathrm{H}}(\mathcal{X},T,d) \leq \sup_{\mu\in \mathscr{M}^T(\mathcal{X})} \underline{\rdim}(\mathcal{X},T,d,\mu). \]
\end{theorem}

\section{Proof of Theorem \ref{theorem: higher rank analogue of Furstenberg}}

First we recall the notations of \S \ref{subsection: statement of the main result}.
 $A^{\mathbb{Z}^2}$ is the $\mathbb{Z}^2$-full shift on the alphabet (finite set) $A$ with the shifts $\sigma_1$ and $\sigma_2$.
Fix $\alpha>1$ and we define the metric $d$ on $A^{\mathbb{Z}^2}$ by 
\[ d(x,y) = \alpha^{-\min \{ |u|_\infty |\, x_u \neq y_u \}}. \]
Let $\mathcal{X}\subset A^{\mathbb{Z}^2}$ be a subshift (closed shift-invariant set) with a Borel probability measure $\mu$ invariant under 
both $\sigma_1$ and $\sigma_2$.

The proof of Theorem \ref{theorem: higher rank analogue of Furstenberg} is divided into 4 steps:
\begin{enumerate}
  \item Prove the upper bound on the upper metric mean dimension
  \[ \overline{\mdim}_{\mathrm{M}}(\mathcal{X}, \sigma_1, d) \leq \frac{2h_{\mathrm{top}}(\mathcal{X}, \sigma_1,\sigma_2)}{\log \alpha}. \]
  \item Prove the lower bound on the lower mean Hausdorff dimension
  \[ \underline{\mdim}_{\mathrm{H}}(\mathcal{X}, \sigma_1, d) \geq \frac{2h_{\mathrm{top}}(\mathcal{X}, \sigma_1,\sigma_2)}{\log \alpha}. \] 
  \item Prove the upper bound on the upper rate distortion dimension
  \[ \overline{\rdim}(\mathcal{X}, \sigma_1,d, \mu) \leq \frac{2h_\mu(\mathcal{X},\sigma_1,\sigma_2)}{\log \alpha}. \]
  \item Prove the lower bound on the lower rate distortion dimension
  \[  \underline{\rdim}(\mathcal{X},\sigma_1,d,\mu) \geq \frac{2 h_{\mu}(\mathcal{X},\sigma_1,\sigma_2)}{\log\alpha}. \]
\end{enumerate}

Since we know $\underline{\mdim}_{\mathrm{H}}(\mathcal{X},\sigma_1,d) \leq \overline{\mdim}_{\mathrm{M}}(\mathcal{X},\sigma_1,d)$
by Proposition \ref{proposition: metric mean dimension bounds mean Hausdorff dimension},
the steps (1) and (2) show 
\[  \mdim_{\mathrm{H}}(\mathcal{X},\sigma_1,d)  = \mdim_{\mathrm{M}}(\mathcal{X}, \sigma_1, d) 
  = \frac{2h_{\mathrm{top}}(\mathcal{X},\sigma_1,\sigma_2)}{\log\alpha}. \]
The steps (3) and (4) show 
\[  \rdim(\mathcal{X},\sigma_1, d,\mu) = \frac{2h_\mu(\mathcal{X},\sigma_1,\sigma_2)}{\log\alpha}. \]
The steps (1) and (3) are easy. 
The step (2) is the most involved.
The four steps are independent of each other.

For $\Omega\subset \mathbb{Z}^2$ we denote by $\pi_\Omega: \mathcal{X}\to A^{\Omega}$ the natural 
projection. 
As in \S \ref{subsection: metric mean dimension and mean Hausdorff dimension}
we set $d_N^{\sigma_1}(x,y) = \max_{0\leq n <N} d(\sigma_1^n x, \sigma_1^n y)$ for $N>0$.
In this section, intervals mean \textit{discrete} intervals.
Namely, for example,  $[a,b] = \{a, a+1, \dots,b-1, b\}$ and $(a,b) = \{a+1,a+2,\dots, b-1\}$ for integers $a\leq b$.

\subsection{Step 1: Proof of 
$\overline{\mdim}_{\mathrm{M}}(\mathcal{X}, \sigma_1, d) \leq 2 h_{\mathrm{top}}(\mathcal{X},\sigma_1,\sigma_2)/\log\alpha$.}

Let $0<\varepsilon<1$ and take a natural number $M$ with $\alpha^{-M} < \varepsilon \leq \alpha^{-M+1}$.
Then
\[  \#(\mathcal{X},d_N^{\sigma_1},\varepsilon) \leq |\pi_{(-M,N+M) \times (-M,M)}(\mathcal{X})|. \]
(Here $|\cdot|$ denotes the cardinality.)
Since $(M-1)\log\alpha \leq \log(1/\varepsilon) < M \log\alpha$, 
\begin{equation*}
  \begin{split}
   \overline{\mdim}_{\mathrm{M}}(\mathcal{X},\sigma_1,d) &= 
   \limsup_{\varepsilon \to 0} \left(\lim_{N\to \infty} \frac{\log \#(\mathcal{X},d_N^{\sigma_1},\varepsilon)}{N\log(1/\varepsilon)}\right) \\
  & \leq  \lim_{M \to \infty} 
  \left(\lim_{N\to \infty} \frac{\log |\pi_{(-M,N+M) \times (-M,M)} (\mathcal{X})|}{N (M-1) \log\alpha} \right) \\
  &   = \frac{2h_{\mathrm{top}}(\mathcal{X},\sigma_1,\sigma_2)}{\log\alpha}.   
   \end{split}
\end{equation*}

\subsection{Step 2: 
Proof of $\underline{\mdim}_{\mathrm{H}}(\mathcal{X},\sigma_1,d) \geq 2 h_{\mathrm{top}}(\mathcal{X},\sigma_1,\sigma_2)/\log\alpha$.}

First we prepare some terminologies about the geometry of $\mathbb{Z}^2$.
In this subsection {\bf rectangles} mean sets of the form $[a,b]\times [c,d]$ in $\mathbb{Z}^2$ for integers $a\leq b$ and $c\leq d$.
For a rectangle $R= [a,b]\times [c,d]$ we define a new rectangle $3R$ by
\begin{equation*}
    3R   = [2a-b, 2b-a] \times [2c-d, 2d-c]. 
\end{equation*}         
We have $|3R| = (3b-3a+1)(3d-3c+1) \leq 9 |R|$.

For two rectangles $R= [a,b]\times [c,d]$ and $R' = [a',b']\times [c',d']$, we denote by $R\leq R'$ if 
$b-a\leq b'-a'$ and $d-c\leq d'-c'$.
This defines an order among rectangles.
(Strictly speaking, this is a ``pre-order'' because $R\leq R'$ and $R'\leq R$ does not imply $R=R'$.)
A set of rectangles $\{R_1,\dots, R_n\}$ is said to be \textbf{totally ordered} if
any two elements are comparable, i.e. for any $R_i$ and $R_j$ we have either $R_i\leq R_j$ 
or $R_j\leq R_i$.

The following trivial fact will be used later:
Suppose $\{R_1, \dots, R_n\}$ is totally ordered. 
If a set of rectangles $\{R'_1,\dots, R'_{n'}\}$ has the property that 
each $R'_i$ is a parallel translation of some $R_j$ (namely $R'_i = u + R_j$ for some $u\in \mathbb{Z}^2$)
then $\{R'_1, \dots, R'_{n'}\}$ is also totally ordered.

The next lemma is a kind of finite Vitali covering lemma (\cite[Lemma 2.27]{Einsiedler--Ward}) adapted to our situation.

\begin{lemma} \label{lemma: covering lemma}
Suppose a set of rectangles $\{R_1, \dots, R_n\}$ is totally ordered.
Then we can find a disjoint subfamily $\{R_{i_1}, \dots, R_{i_m}\}$ satisfying
\[   R_1\cup \dots \cup R_n  \subset 3R_{i_1}\cup 3R_{i_2} \cup \dots \cup 3R_{i_m}. \]
Note that this implies 
\[  |R_{i_1}\cup \dots \cup R_{i_m}| \geq \frac{1}{9} |R_1\cup \dots \cup R_n|. \]
\end{lemma}

\begin{proof}
We use a simple greedy algorithm.
We first choose (one of) the largest rectangle, say $R_{i_1}$.
Next, suppose we have chosen $R_{i_1}, \dots, R_{i_k}$.
We choose as $R_{i_{k+1}}$ the largest rectangle disjoint to $R_{i_1}\cup \dots \cup R_{i_k}$.
If there is no such a rectangle, the algorithm stops.

Suppose the algorithm stops after $m$ steps.
For any $R_j$ there exists $R_{i_k}$ with $R_{i_k}\geq R_j$ and $R_{i_k}\cap R_j \neq \emptyset$.
This implies $R_j\subset 3R_{i_k}$.
\end{proof}

For two sets $\Omega, \Lambda\subset \mathbb{Z}^2$ we define $\partial_\Lambda \Omega$ as the set of 
$u\in \mathbb{Z}^2$ such that $u+\Lambda$ has non-empty intersections both with $\Omega$ and $\mathbb{Z}^2\setminus \Omega$.
We set $\mathrm{Int}_\Lambda \Omega = \Omega\setminus \partial_\Lambda \Omega$.
This is the set of $u\in \Omega$ with $u+\Lambda \subset \Omega$.

Let $R\subset \mathbb{Z}^2$ be a rectangle.
A subset $C\subset \mathcal{X}$ is called a \textbf{cylinder over $R$} if there is $x\in \mathcal{X}$ such that 
$C$ is equal to the set of $y\in \mathcal{X}$ satisfying $\pi_R(y)  = \pi_R(x)$. 

Set 
\[   s = \frac{2h_{\mathrm{top}}(\mathcal{X},\sigma_1,\sigma_2)}{\log\alpha}. \]
Suppose $\underline{\mdim}_{\mathrm{H}}(\mathcal{X},\sigma_1,d) < s$. We would like to get a contradiction.
We fix $\varepsilon>0$ satisfying $\underline{\mdim}_{\mathrm{H}}(\mathcal{X},\sigma_1,d) < s-2\varepsilon$.

\begin{lemma}  \label{lemma: finding rectangles}
For any finite subset $\Lambda \subset \mathbb{Z}^2$ and any positive number $L$, we can find rectangles
$R_1, \dots, R_M \subset \mathbb{Z}^2$ and subsets $C_1,\dots, C_M\subset \mathcal{X}$ such that
\begin{itemize}
   \item Each $C_m$ is a cylinder over $R_m$ and they satisfy $\mathcal{X} = \bigcup_{m=1}^M C_m$.
   \item All the rectangles $R_m$ contain the origin, and they are all sufficiently large so that 
           \[  |\partial_\Lambda R_m| < \frac{|R_m|}{L}, \quad |R_m| > L. \]
   \item The rectangles $R_1, \dots, R_M$ are totally ordered and satisfy         
   \[  \sum_{m=1}^M \alpha^{-\frac{1}{2}(s-\varepsilon)|R_m|} < 1. \]
\end{itemize}
\end{lemma}

\begin{proof}
We choose a natural number $r_0$ such that 
\begin{itemize}
   \item Every $r\geq r_0$ satisfies $(s-2\varepsilon)r < (s-\varepsilon)(r-1)$.
   \item If a rectangle $R=[a,b]\times [c,d]\subset \mathbb{Z}^2$ satisfies $b-a\geq r_0$ and $d-c\geq r_0$ then 
          \[  |\partial_\Lambda R| < \frac{|R|}{L},  \quad |R|>L. \]
\end{itemize}
From $\underline{\mdim}_{\mathrm{H}}(\mathcal{X},\sigma_1,d) < s-2\varepsilon$, we can find $N>0$ satisfying 
\[  \frac{1}{N} \dim_{\mathrm{H}}(\mathcal{X}, d^{\sigma_1}_N, \alpha^{-r_0}) < s-2\varepsilon. \]
This implies that there exists a covering $\mathcal{X} = E_1\cup \dots \cup E_M$ satisfying 
\[  \diam (E_m, d^{\sigma_1}_N) < \alpha^{-r_0} \> (\forall 1\leq m \leq M), \quad 
    \sum_{m=1}^M \left(\diam(E_m, d^{\sigma_1}_N)\right)^{(s-2\varepsilon)N} < 1. \]
Set $\alpha^{-r_m} := \diam (E_m,d^{\sigma_1}_N)$.
Then $r_m$ is a natural number with $r_m>r_0$.
Choose a point $x_m$ from each $E_m$, and let $C_m\subset \mathcal{X}$ be a cylinder over the rectangle
\[  R_m := [-r_m+1, N+r_m-2]\times [-r_m+1, r_m-1] \]
defined by $C_m = \pi_{R_m}^{-1}(\pi_{R_m}(x_m))$.
Then $E_m\subset C_m$ and hence $\mathcal{X} = C_1\cup \dots \cup C_m$.
The rectangles $R_m$ are totally ordered ($R_m\leq R_{m'}$ if and only if $r_m\leq r_{m'}$).

Recall that $r_m >r_0$ for all $1\leq m \leq M$.
From the choice of $r_0$, 
\[  |\partial_\Lambda R_m| <  \frac{|R_m|}{L}, \quad |R_m| > L.   \]
From $|R_m| = (N+2r_m-2) (2r_m-1) \geq  N(2r_m-1)$, 
\begin{equation*}
   \begin{split}
   \frac{1}{2}(s-\varepsilon) |R_m|  & \geq \frac{1}{2}(s-\varepsilon) N (2r_m-1) \\
          &  >  \frac{1}{2} (s-2\varepsilon)N (2r_m)   \quad 
          \text{by the choice of $r_0$} \\
          & = (s-2\varepsilon) N r_m. 
   \end{split}
\end{equation*}   
Hence 
\[  \alpha^{-\frac{1}{2}(s-\varepsilon)|R_m|} < \alpha^{-(s-2\varepsilon)N r_m} 
    =  \left(\diam(E_m, d^{\sigma_1}_N)\right)^{(s-2\varepsilon)N} .\]
Therefore 
\[   \sum_{m=1}^M \alpha^{-\frac{1}{2}(s-\varepsilon)|R_m|}  
      <    \sum_{m=1}^M \left(\diam(E_m, d^{\sigma_1}_N)\right)^{(s-2\varepsilon)N} < 1. \]
\end{proof}

We choose a real number $0<\delta<1/2$ and a natural number $p$ satisfying the following conditions.
\begin{equation}  \label{eq: choices of delta and p}
  \left(\frac{17}{18}\right)^{p} < \delta, \quad
 H(\delta) + \delta \log p < \frac{\varepsilon}{8} \log \alpha, \quad 
 |A|^\delta < \alpha^{\varepsilon/8}.
\end{equation}
Here $H(\delta) = -\delta \log \delta - (1-\delta)\log (1-\delta)$.
(Recall that the base of the logarithm is two.)
The first condition is satisfied for $p  \approx \log (1/\delta)$.
Then we choose a sufficiently small $\delta$ satisfying the second and third conditions.

By using Lemma \ref{lemma: finding rectangles} iteratively, we find rectangles $R_{i, m}$ and subsets $C_{i, m}\subset \mathcal{X}$
for $i=1, \dots, p$ and $m=1,\dots, M_i$ (where $M_i$ is a natural number depending on $i$) satisfying the following conditions.

\begin{enumerate}
   \item[(a)] Each $C_{i, m}$ is a cylinder over $R_{i, m}$. For each $1\leq i\leq p$ we have $\mathcal{X}= \bigcup_{m=1}^{M_i} C_{i,m}$.
   \item[(b)] For each $1\leq i\leq p$, the rectangles $R_{i, 1}, R_{i,2}, \dots, R_{i, M_i}$ are totally ordered and satisfy
   \begin{equation}  \label{eq: too efficient covering}
        \sum_{m=1}^{M_i} \alpha^{-\frac{1}{2}(s-\varepsilon)|R_{i, m}|} < 1.   
   \end{equation}
   \item[(c)]  All the rectangles $R_{i,m}$ contain the origin and they satisfy $|R_{i, m}| > 1/\delta$.
   \item[(d)]  Set $\hat{R}_i = \bigcup_{m=1}^{M_i} R_{i, m}$. Then for all $j<i$ and $m=1,\dots, M_j$ we have 
   \[     |\partial_{\hat{R}_i} R_{j, m}| < \frac{\delta}{4} |R_{j, m}|. \]
\end{enumerate}

Roughly speaking, the condition (d) means that the rectangles in one level (say, $j$) are much larger than the
rectangles in higher levels (say, $i>j$). 
The construction goes from the level $p$ to the bottom.
First, by Lemma \ref{lemma: finding rectangles}, we construct $R_{p, m}$ and $C_{p,m}$.
Next, by using the lemma again, we construct $R_{p-1, m}$ and $C_{p-1,m}$.
We continue this process until we come to the first level ($R_{1,m}$ and $C_{1,m}$). 
The condition (d) connects the constructions in different levels.

\begin{lemma} \label{lemma: iterated covering lemma}
If $N>0$ is sufficiently large then the following statement holds.
For each $x\in \mathcal{X}$ we can choose a subset
\[   D(x)\subset \{(u, i, m)|\, u\in [0,N-1]^2, 1\leq i\leq p, 1\leq m\leq M_i\} \] 
such that 
  \begin{enumerate}
     \item For $(u,i,m)\in D(x)$, we have $\sigma^u(x)\in C_{i,m}$
             and $u+R_{i,m} \subset [0,N-1]^2$.
     \item If $(u, i, m)$ and $(u', i', m')$ are two different elements of $D(x)$, then $(u+R_{i,m}) \cap (u'+R_{i',m'}) = \emptyset$.
     In particular (recall that $R_{i,m}$ contain the origin), $u\neq u'$.
     \item We have 
     \[   \left|[0,N-1]^2 \setminus \bigcup_{(u,i,m)\in D(x)} (u+R_{i,m})\right|  < \delta N^2. \]
  \end{enumerate}
\end{lemma}

\begin{proof}
Let $N$ be sufficiently large so that 
\begin{equation} \label{eq: choosing large N}
   \left|\partial_{\hat{R}_i} [0,N-1]^2\right| < \frac{\delta}{4} N^2 \quad   \text{for all $1\leq i \leq p$}.
\end{equation}
Here recall that $\hat{R}_i = \bigcup_{m=1}^{M_i} R_{i,m}$.
Fix $x\in \mathcal{X}$.
Set $E_0 = [0,N-1]^2$.
We will inductively construct $E_0\supset E_1\supset E_2 \supset \dots \supset E_p$.

Suppose we have defined $E_0, E_1, \dots, E_{i-1}$.  Consider the following set of rectangles:
\begin{equation}  \label{eq: covering by rectangles}
 \left\{u+ R_{i, m}|\, u\in E_{i-1}, m\in [1, M_i]  \text{ with } \sigma^u(x)\in C_{i, m} \text{ and }  u+R_{i, m} \subset E_{i-1}\right\}. 
\end{equation} 
Since $R_{i, 1}, \dots, R_{i, M_i}$ are totally ordered, so is (\ref{eq: covering by rectangles}). (Here the point is that $i$ is fixed.)
The rectangles (\ref{eq: covering by rectangles}) cover $\mathrm{Int}_{\hat{R}_i} E_{i-1}$. 
Then by Lemma \ref{lemma: covering lemma}, we can find a subset 
\[  D_i(x) \subset \left\{(u, m)|\, u\in E_{i-1}, 1\leq m\leq M_i \right\} \]
such that 
\begin{itemize}
   \item For $(u,m)\in D_i(x)$, we have $\sigma^u(x)\in C_{i, m}$ and $u+R_{i,m}\subset E_{i-1}$.
   \item If $(u,m)$ and $(u',m')$ are two different elements of $D_i(x)$ then $(u+R_{i,m}) \cap (u'+R_{i, m'}) = \emptyset$.
   \item The rectangles $u+R_{i,m}$, $(u,m)\in D_i(x)$, cover at least one-ninth of $\mathrm{Int}_{\hat{R}_i} E_{i-1}$:
   \begin{equation}  \label{eq: one-ninth is covered}
     \left|\bigcup_{(u,m)\in D_i(x)} (u+ R_{i,m}) \right| \geq \frac{1}{9} \left|\mathrm{Int}_{\hat{R}_i} E_{i-1}\right|.
   \end{equation}
\end{itemize}
We set 
\[  E_i = E_{i-1}\setminus \bigcup_{(u, m)\in D_i(x)} (u+R_{i,m}). \]

We define $D(x)$ by 
\[  D(x) = \left\{(u, i, m)|\, 1\leq i\leq p, (u, m) \in D_i(x)\right\}. \]
The properties (1) and (2) of $D(x)$ immediately follow from the construction.
The property (3) is equivalent to the claim that $|E_{p}| < \delta N^2$. 
We will prove this.

Suppose $|E_{p}|\geq \delta N^2$.
Then we also have $|E_{i-1}| \geq \delta N^2$ for all $1\leq i\leq p$.
We estimate $\left|\partial_{\hat{R}_i} E_{i-1}\right|$ for $1\leq i\leq p$.
We have 
\[  \partial_{\hat{R}_i} E_{i-1} \subset \partial_{\hat{R}_i} [0,N-1]^2 \cup 
     \bigcup_{j=1}^{i-1} \bigcup_{(u, m)\in D_j(x)} \partial_{\hat{R}_i} (u+R_{j, m}). \]
Recall (\ref{eq: choosing large N}) and $|\partial_{\hat{R}_i} R_{j,m}| < (\delta/4) |R_{j,m}|$ for 
$j < i$ by the condition (d) of the choice of $R_{i,m}$.
Then 
\begin{equation*}
  \begin{split}
    \left|\partial_{\hat{R}_i} E_{i-1}\right| & \leq \left|\partial_{\hat{R}_i} [0,N-1]^2\right| 
    + \sum_{j=1}^{i-1} \sum_{(u,m)\in D_j(x)} \left| \partial_{\hat{R}_i} (u+R_{j, m})\right| \\
    & <     \frac{\delta}{4} N^2 + \frac{\delta}{4}  \sum_{j=1}^{i-1} \sum_{(u,m)\in D_j(x)} \left|u+R_{j, m}\right|.
   \end{split} 
\end{equation*}    
The rectangles $u+R_{j,m}$, $1\leq j\leq i-1$ and $(u,m)\in D_j(x)$, are disjoint and contained in $[0,N-1]^2$.
Therefore 
\[  \sum_{j=1}^{i-1} \sum_{(u,m)\in D_j(x)} \left|u+R_{j, m}\right|  \leq N^2. \]
Thus $\left|\partial_{\hat{R}_i} E_{i-1}\right|  < (\delta/2)N^2$.
Since we assumed $|E_{i-1}|\geq \delta N^2$, we have $\left|\partial_{\hat{R}_i} E_{i-1}\right|  <(1/2) |E_{i-1}|$.
Namely 
\[   \left|\mathrm{Int}_{\hat{R}_i} E_{i-1}\right| > \frac{1}{2} |E_{i-1}|. \]

From (\ref{eq: one-ninth is covered}),
\[    \left|\bigcup_{(u,m)\in D_i(x)} (u+ R_{i,m}) \right| \geq \frac{1}{9} \left|\mathrm{Int}_{\hat{R}_i} E_{i-1}\right| 
       > \frac{1}{18} |E_{i-1}|. \]
So we get 
\[  |E_i| = \left|E_{i-1}\setminus \bigcup_{(u, m)\in D_i(x)} (u+R_{i,m})\right| < \frac{17}{18} |E_{i-1}|. \]
This holds for all $1\leq i\leq p$. Therefore 
\[  |E_p| < \left(\frac{17}{18}\right)^p |E_0| = \left(\frac{17}{18}\right)^p N^2. \]
Recall that $p$ satisfies $(17/18)^p < \delta$ by (\ref{eq: choices of delta and p}).
So $|E_p| < \delta N^2$. This is a contradiction.
\end{proof}

In the rest of this subsection, $N$ is assumed to be so large that the statement of Lemma \ref{lemma: iterated covering lemma} holds.
For each $x\in \mathcal{X}$ we define $\underline{D}(x)\subset [0,N-1]^2\times [1, p]$
as the set of $(u, i) \in [0,N-1]^2\times [1,p]$ such that there exists $m\in [1,M_i]$ with $(u, i, m)\in D(x)$.
(Notice that the sets $D(x)$ and $\underline{D}(x)$ depend on $N$. So it might be better to use the notations
$D^{(N)}(x)$ and $\underline{D}^{(N)}(x)$. But we prefer the simpler ones here.)

\begin{lemma}  \label{lemma: possibilities of underline_D}
If $N$ is sufficiently large then the number of possibilities of $\underline{D}(x)$ is bounded as follows:
\[  \left|\left\{\underline{D}(x) |\, x\in \mathcal{X} \right\}\right|    < \alpha^{(\varepsilon/8) N^2}. \]
\end{lemma}

\begin{proof}
We use the well-known bound on the binomial coefficient:
\begin{equation} \label{eq: bound on binom}
     \binom{n}{k} \leq 2^{n H(k/n)}. 
\end{equation}     
This follows from 
\[  1= \left\{\frac{k}{n} + \left(1-\frac{k}{n}\right)\right\}^n \geq 
    \binom{n}{k}  \left(\frac{k}{n}\right)^{k} \left(1-\frac{k}{n}\right)^{n-k} = \binom{n}{k} 2^{-n H(k/n)}. \]

Let $x\in \mathcal{X}$ and set $D(x) = \{(u_1, i_1, m_1), \dots, (u_k, i_k, m_k)\}$.
(Then we have $\underline{D}(x) = \{(u_1, i_1), \dots, (u_k, i_k)\}$.)
By (1) and (2) of Lemma \ref{lemma: iterated covering lemma}, $u_1, \dots, u_k$ are different from each other, and 
the rectangles $u_1+R_{i_1,m_1}, \dots, u_k+R_{i_k, m_k}$ are disjoint and contained in $[0,N-1]^2$.
Since $|R_{i, m}|> 1/\delta$ (the condition (c) of the choice of $R_{i,m}$), we have $k < \delta N^2$.

Then the number of possibilities of $\underline{D}(x)$ is bounded by 
\begin{equation*}
   \begin{split}
   & \underbrace{\left\{\binom{N^2}{1} + \binom{N^2}{2} + \dots + \binom{N^2}{\lfloor \delta N^2\rfloor}\right\}}_{\text{choices of $u_1, \dots, u_k$}}
    \times \underbrace{p^{\delta N^2}}_{\text{choices of $i_1, \dots, i_k$}} \\
   & \leq  N^2\cdot 2^{N^2 H(\delta)} \times p^{\delta N^2}    \quad 
      \text{by (\ref{eq: bound on binom})} \\
   & = N^2\cdot 2^{N^2\left(H(\delta) + \delta \log p\right)}. 
   \end{split} 
\end{equation*}   
We assumed $H(\delta)+\delta\log p < (\varepsilon/8)\log \alpha$ in (\ref{eq: choices of delta and p}).
Hence, if $N$ is sufficiently large then 
\[   N^2\cdot 2^{N^2\left(H(\delta) + \delta \log p\right)}  < 2^{N^2 (\varepsilon/8) \log\alpha} = \alpha^{(\varepsilon/8)N^2}. \]
\end{proof}

Take a subset $E\subset [0,N-1]^2\times [1,p]$ such that there exists $x\in \mathcal{X}$ with $\underline{D}(x)=E$.
We denote by $\mathcal{X}_E$ the set of $x\in \mathcal{X}$ with $\underline{D}(x)=E$.
Let $E=\{(u_1, i_1), (u_2, i_2), \dots, (u_k, i_k)\}$.

\begin{lemma}
  \begin{equation}  \label{eq: bounding the cardinality of X_E}
    \begin{split}
     & \left|\pi_{[0,N-1]^2}(\mathcal{X}_E)\right|  \cdot \alpha^{-\frac{1}{2}(s-\varepsilon)N^2}  \\
        & \leq  |A|^{\delta N^2}  \left(\sum_{m=1}^{M_{i_1}} \alpha^{-\frac{1}{2}(s-\varepsilon)|R_{i_1, m}|} \right) \times \dots \times 
           \left(\sum_{m=1}^{M_{i_k}} \alpha^{-\frac{1}{2}(s-\varepsilon)|R_{i_k, m}|}\right).
    \end{split}
  \end{equation}
\end{lemma}

\begin{proof}
For ${\bf m} = (m_1, \dots, m_k) \in [1, M_{i_1}]\times \dots \times [1, M_{i_k}]$, we denote by $\mathcal{X}_{E,{\bf m}}\subset \mathcal{X}_E$
the set of $x\in \mathcal{X}_E$ with $D(x) = \{(u_1, i_1, m_1), (u_2, i_2, m_2), \dots, (u_k, i_k, m_k)\}$.
We have $\sigma^{u_j}(x) \in C_{i_j, m_j}$ for $x\in \mathcal{X}_{E,{\bf m}}$.
Hence, over each rectangle $u_j + R_{i_j, m_j}$, the value of $\pi_{u_j + R_{i_j, m_j}}(x)$ $(x\in \mathcal{X}_{E, {\bf m}})$ is fixed.
(Namely we have $\pi_{u_j + R_{i_j, m_j}}(x) = \pi_{u_j + R_{i_j, m_j}}(x')$ for any two $x, x'\in \mathcal{X}_{E,{\bf m}}$.)
Therefore we have 
\[  |\pi_{[0, N-1]^2}\left(\mathcal{X}_{E, {\bf m}}\right)| \leq |A|^{\left|[0, N-1]^2\setminus \bigcup_{j=1}^k (u_j+R_{i_j, m_j})\right|} 
      <  |A|^{\delta N^2}. \]
Here the second inequality follows from the condition (3) of Lemma \ref{lemma: iterated covering lemma}.
We decompose the left-hand side of (\ref{eq: bounding the cardinality of X_E}) as 
\begin{equation} \label{eq: decomposing X_E}
    \begin{split}
     \left|\pi_{[0,N-1]^2}(\mathcal{X}_E)\right|  \cdot \alpha^{-\frac{1}{2}(s-\varepsilon)N^2}  
     & = \sum_{{\bf m}}   \left|\pi_{[0,N-1]^2}(\mathcal{X}_{E,{\bf m}})\right|  \cdot \alpha^{-\frac{1}{2}(s-\varepsilon)N^2} \\
     & \leq \sum_{{\bf m} \text{ with } \mathcal{X}_{E,{\bf m}} \neq \emptyset} |A|^{\delta N^2} \cdot \alpha^{-\frac{1}{2}(s-\varepsilon)N^2} . 
    \end{split}  
\end{equation}      
Take ${\bf m}  = (m_1, \dots, m_k)\in [1, M_{i_1}]\times \dots \times [1, M_{i_k}]$ with $\mathcal{X}_{E,{\bf m}} \neq \emptyset$. 
The rectangles $u_j + R_{i_j, m_j}$ $(1\leq j\leq k)$ are disjoint
and contained in $[0,N-1]^2$ by the conditions (1) and (2) of Lemma \ref{lemma: iterated covering lemma}.
Hence 
\[  N^2 \geq \sum_{j=1}^k |R_{i_j, m_k}|. \]
So 
\[  \alpha^{-\frac{1}{2}(s-\varepsilon)N^2} \leq \prod_{j=1}^k \alpha^{-\frac{1}{2}(s-\varepsilon)|R_{i_j, m_j}|}. \]
Plugging this into (\ref{eq: decomposing X_E}), we get 
\begin{equation*}
       \left|\pi_{[0,N-1]^2}(\mathcal{X}_E)\right|  \cdot \alpha^{-\frac{1}{2}(s-\varepsilon)N^2}  \leq
     \sum_{{\bf m}} |A|^{\delta N^2}  \prod_{j=1}^k \alpha^{-\frac{1}{2}(s-\varepsilon)|R_{i_j, m_j}|}.
\end{equation*}     
The right-hand side is equal to 
\[   |A|^{\delta N^2}   \left(\sum_{m=1}^{M_{i_1}} \alpha^{-\frac{1}{2}(s-\varepsilon)|R_{i_1, m}|} \right) \times \dots \times 
           \left(\sum_{m=1}^{M_{i_k}} \alpha^{-\frac{1}{2}(s-\varepsilon)|R_{i_k, m}|}\right). \]
\end{proof}

We continue the estimates:

\begin{equation*}
   \begin{split}
    \left|\pi_{[0,N-1]^2}(\mathcal{X}_E)\right|  \cdot \alpha^{-\frac{1}{2}(s-\varepsilon)N^2}  
        & \leq  |A|^{\delta N^2}   \left(\sum_{m=1}^{M_{i_1}} \alpha^{-\frac{1}{2}(s-\varepsilon)|R_{i_1, m}|} \right) \times \dots \times 
           \left(\sum_{m=1}^{M_{i_k}} \alpha^{-\frac{1}{2}(s-\varepsilon)|R_{i_k, m}|}\right)   \\
     & <  |A|^{\delta N^2}  \quad \text{by (\ref{eq: too efficient covering})} \\
     & <  \alpha^{(\varepsilon/8)N^2}   \quad 
     \text{since we assumed $|A|^\delta < \alpha^{\varepsilon/8}$ in (\ref{eq: choices of delta and p})}.
   \end{split}         
\end{equation*}

The number of choices of $E\subset [0,N-1]^2\times [1,p]$ with $\mathcal{X}_E \neq \emptyset$ 
is bounded by $\alpha^{(\varepsilon/8)N^2}$
if $N$ is sufficiently large
(Lemma \ref{lemma: possibilities of underline_D}).
Then
\begin{equation*}
   \begin{split}
        \left|\pi_{[0,N-1]^2}(\mathcal{X})\right|  \cdot \alpha^{-\frac{1}{2}(s-\varepsilon)N^2} 
      & = \sum_{E \text{ with } \mathcal{X}_E\neq \emptyset}
         \left|\pi_{[0,N-1]^2}(\mathcal{X}_{E})\right|  \cdot \alpha^{-\frac{1}{2}(s-\varepsilon)N^2} \\
      & <   \alpha^{(\varepsilon/8)N^2} \times \alpha^{(\varepsilon/8)N^2}  = \alpha^{(\varepsilon/4)N^2}. 
    \end{split}
\end{equation*}       
Therefore 
\[    \left|\pi_{[0,N-1]^2}(\mathcal{X})\right|   < \alpha^{\frac{1}{2}\left(s-\frac{\varepsilon}{2}\right)N^2}. \]
Namely 
\[  \frac{\log |\pi_{[0,N-1]^2}(\mathcal{X})|}{N^2}  < \frac{1}{2}\left(s-\frac{\varepsilon}{2}\right)\log\alpha. \]
Letting $N\to \infty$
\[  h_{\mathrm{top}}(\mathcal{X},\sigma_1,\sigma_2) \leq \frac{1}{2}\left(s-\frac{\varepsilon}{2}\right)\log\alpha
     < \frac{1}{2}s \log\alpha = h_{\mathrm{top}}(\mathcal{X},\sigma_1,\sigma_2). \]
This is a contradiction.

\begin{remark}
 \begin{enumerate}
   \item The above proof (in particular, see the proof of Lemma \ref{lemma: finding rectangles})
also shows a (seemingly) slightly stronger statement that
\[  \lim_{\varepsilon \to 0} \left(\inf_{N\geq 1} \frac{\dim_{\mathrm{H}}(\mathcal{X},d^{\sigma_1}_N,\varepsilon)}{N}\right)   
     \geq \frac{2h_{\mathrm{top}}(\mathcal{X},\sigma_1,\sigma_2)}{\log \alpha}. \]     
Combined with Step 1, the both sides actually coincide.
However we do not know whether the left-hand side is an important quantity or not.     
   \item The above proof (in particular, the use of covering argument) is motivated by the proof of the Shannon--McMillan--Breiman 
   theorem (see, e.g. \cite{Ornstein--Weiss, Rudolph, Lindenstrauss pointwise}).
   We expect that there is a proof more directly using the Shannon--McMillan--Breiman theorem (or related measure theoretic ideas)
    although we have not found it so far.
  \end{enumerate} 
\end{remark}

\subsection{Step 3: Proof of $\overline{\rdim}(\mathcal{X},\sigma_1,d,\mu) \leq 2 h_{\mu}(\mathcal{X},\sigma_1,\sigma_2)/\log\alpha$.}

Let $X$ be a random variable taking values in $\mathcal{X}$ and obeying $\mu$.
Let $0<\varepsilon<1$ and take $M>0$ with $\alpha^{-M} < \varepsilon \leq \alpha^{-M+1}$ as in Step 1.
Let $N>0$.
For each point $x\in \pi_{(-M,N+M) \times (-M,M)}(\mathcal{X})$ we choose $q(x)\in \mathcal{X}$ 
with $\pi_{(-M,N+M) \times (-M,M)}(q(x))=x$.
Set $X' = q\left(\pi_{(-M,N+M) \times (-M,M)}(X)\right)$ and 
$Y = (X', \sigma_1 X', \sigma_1^2 X', \dots, \sigma_1^{N-1}X')$.
Then 
\[  \frac{1}{N} \sum_{n=0}^{N-1}d(\sigma_1^n X, Y_n) = \frac{1}{N}\sum_{n=0}^{N-1}d(\sigma_1^n X, \sigma_1^n X') 
     \leq \alpha^{-M} < \varepsilon. \]
\[  I(X;Y) \leq H(Y) =H(X') = H\left\{ (X_u)_{u\in (-M,N+M)\times (-M,M)}\right\}. \]
So 
\[  R(d,\mu,\varepsilon) \leq \frac{I(X;Y)}{N} \leq \frac{1}{N}  H\left\{ (X_u)_{u\in (-M,N+M)\times (-M,M)}\right\},  \]
\[  \frac{R(d,\mu,\varepsilon)}{\log(1/\varepsilon)}  \leq \frac{2M}{\log(1/\varepsilon)} \cdot 
     \frac{1}{2NM} H\left\{ (X_u)_{u\in (-M,N+M)\times (-M,M)}\right\}.  \]
We first take the limit with respect to $N$ and next the limit with respect to $\varepsilon$.
Noting $M/\log(1/\varepsilon)\to 1/\log\alpha$, we get
\[  \overline{\rdim}(\mathcal{X},\sigma_1,d,\mu) \leq \frac{2 h_{\mu}(\mathcal{X},\sigma_1,\sigma_2)}{\log\alpha}. \]

\subsection{Step 4: Proof of $\underline{\rdim}(\mathcal{X},\sigma_1,d,\mu) \geq 2 h_{\mu}(\mathcal{X},\sigma_1,\sigma_2)/\log\alpha$.}

We need the following lemma.

\begin{lemma}  \label{lemma: if X and Y coincide at many sites then I(X;Y) is large}
    Let $N\geq 1$ and $B$ a finite set.
    Let $X=(X_0,\dots, X_{N-1})$ and $Y=(Y_0,\dots, Y_{N-1})$ be random variables taking values in $B^N$ 
     (namely, each $X_n$ and $Y_n$ takes values in $B$) such that for some $0<\delta <1/2$
     \[  \mathbb{E}\left(\text{the number of $0\leq n < N$ with $X_n\neq Y_n$}\right) < \delta N. \]
     Then 
     \[  I(X;Y) > H(X) - N H(\delta) -\delta N \log |B|, \]
     where $H(\delta) = -\delta \log \delta - (1-\delta)\log (1-\delta)$ as in Step 2.
\end{lemma}

\begin{proof}
The proof is close to \cite[Lemma 17]{Lindenstrauss--Tsukamoto rate distortion}.
Let $Z_n = 1_{\{X_n\neq Y_n\}}$ and $Z=\{0\leq n <N|\, X_n\neq Y_n\}$.
We can identify $Z$ with $(Z_0,\dots,Z_{N-1})$ and hence 
\begin{equation*}
   \begin{split}
    H(Z) &\leq H(Z_0)+\dots+H(Z_{N-1}) \\
           &= H\left(\mathbb{E} Z_0 \right) +\dots+ H\left(\mathbb{E} Z_{N-1}\right)  \\
           &\leq N H\left(\frac{1}{N} \sum_{n=0}^{N-1}\mathbb{E} Z_n\right)   \quad (\text{by concavity of $H(\cdot)$})\\
           &< N H(\delta).
   \end{split}
\end{equation*}   
So $H(Z) < NH(\delta)$. We decompose $H(X,Z|Y)$ in two ways:
\[ H(X,Z|Y) = H(X|Y) + H(Z|X,Y) = H(Z|Y) + H(X|Y,Z). \]
$H(Z|X,Y) = 0$ because $Z$ is determined by $X$ and $Y$.
Hence 
\[  H(X|Y) = H(Z|Y) + H(X|Y,Z) < N H(\delta) + H(X|Y,Z). \]
We estimate 
\[  H(X|Y,Z) = \sum_{E\subset \{0,1,\dots,N-1\}}  \mathbb{P}(Z=E) H(X|Y,Z=E). \]
Given $Y$ and the condition $Z=E$, the possibilities of $X$ is at most $|B|^{|E|}$.
Therefore $H(X|Y,Z=E) \leq |E| \log |B|$ and 
\begin{equation*}
   \begin{split}
       H(X|Y,Z) &\leq 
       \sum_{E\subset \{0,1,\dots,N-1\}}  |E|\cdot \mathbb{P}(Z=E) \log |B| \\
       & = \mathbb{E} |Z| \cdot \log |B| \\
       & \leq \delta N \log |B|.
   \end{split}
\end{equation*}       
As a conclusion, $H(X|Y) < N H(\delta) + \delta N \log |B|$ and
$I(X;Y) = H(X)-H(X|Y) > H(X) - N H(\delta) - \delta N \log |B|$.
\end{proof}

Let $X$ be a random variable taking values in $\mathcal{X}$ with $\mathrm{Law}(X) = \mu$ as in Step 3.
Let $0<\varepsilon<\delta<1/2$ and $N>0$.
Let $Y=(Y_0,\dots,Y_{N-1})$ be a random variable taking values in $\mathcal{X}^N$ and satisfying 
\[   \mathbb{E}\left(\frac{1}{N}\sum_{n=0}^{N-1} d(\sigma_1^n X, Y_n)\right) < \varepsilon. \]
We estimate $I(X;Y)$ from below.
Take $M\geq 0$ satisfying $\delta\alpha^{-M-1} < \varepsilon \leq \delta \alpha^{-M}$.
For $0\leq n <N$, we set 
\[  X'_n = \pi_{\{n\}\times [-M,M]}(X)  = (X_{n,m})_{-M\leq m\leq M}, \quad 
    Y'_n = \pi_{\{0\}\times [-M,M]}(Y_n) =   \left((Y_n)_{0,m}\right)_{-M\leq m\leq M}. \]
If $X'_n\neq Y'_n$ for some $n$ then $d(\sigma_1^n X, Y_n) \geq \alpha^{-M}$.
So $\mathbb{E}d(\sigma_1^n X, Y_n) \geq \alpha^{-M} \mathbb{P}(X'_n\neq Y'_n)$ and hence 
\begin{equation*}
    \begin{split}
       \mathbb{E}\left(\text{the number of $0\leq n < N$ with $X'_n\neq Y'_n$}\right)& = \sum_{n=0}^{N-1} \mathbb{P}(X'_n\neq Y'_n) \\
       &\leq \alpha^M \mathbb{E}\left(\sum_{n=0}^{N-1}d(\sigma_1^n X, Y_n)\right) \\
       &<  \alpha^M \varepsilon  N  \leq \delta N.
    \end{split}
\end{equation*}    
Apply Lemma \ref{lemma: if X and Y coincide at many sites then I(X;Y) is large} to $X'_n$ and $Y'_n$ with 
$B=A^{2M+1}$:
\begin{equation*}
   \begin{split}
     I(X'_0,\dots,X'_{N-1}; Y'_0,\dots, Y'_{N-1}) > & H(X'_0,\dots,X'_{N-1}) \\
     &-N H(\delta) -\delta N (2M+1) \log |A|. 
    \end{split}
\end{equation*}     
By the data-processing inequality (Lemma \ref{lemma: data-processing inequality}),
\[   I(X;Y) \geq   I(X'_0,\dots,X'_{N-1}; Y'_0,\dots, Y'_{N-1}). \]
Therefore 
\[  \frac{I(X;Y)}{N} \geq \frac{H\left\{(X_u)_{u\in [0,N)\times [-M,M]}\right\}}{N} -H(\delta) -\delta(2M+1)\log|A|. \]
This holds for any $N>0$. So 
\begin{equation*}
    \begin{split}
      R(d,\mu,\varepsilon) & \geq \inf_{N>0} \frac{H\left\{(X_u)_{u\in [0,N)\times [-M,M]}\right\}}{N}  -H(\delta) -\delta(2M+1)\log|A| \\
      &  = \lim_{N\to \infty} \frac{H\left\{(X_u)_{u\in [0,N)\times [-M,M]}\right\}}{N} -H(\delta) -\delta(2M+1)\log|A|.
    \end{split}
\end{equation*}    
We divide this by $\log(1/\varepsilon)$ and take the limit $\varepsilon \to 0$.
Noting $\log(1/\varepsilon) < \log(1/\delta) + (M+1) \log\alpha$ (here $\delta$ has been fixed), we get
\[  \underline{\rdim}(\mathcal{X},\sigma_1,d,\mu) \geq \frac{2h_\mu(\mathcal{X},\sigma_1,\sigma_2)}{\log\alpha}
     -\frac{2\delta \log|A|}{\log\alpha}. \]
Here we have used 
\[ h_\mu(\mathcal{X},\sigma_1,\sigma_2) = \lim_{N,M\to \infty} \frac{H\left\{(X_u)_{u\in [0,N)\times [-M,M]}\right\}}{N(2M+1)}. \]
Take the limit $\delta\to 0$.
We get $\underline{\rdim}(\mathcal{X},\sigma_1,d,\mu) \geq 2 h_{\mu}(\mathcal{X},\sigma_1,\sigma_2)/\log\alpha$.

\vspace{0.5cm}

\address{ Mao Shinoda \endgraf
Department of Human Coexistence, Graduate School of Human and Environmental Studies, Kyoto University, Yoshida-Nihonmaths-cho, Sakyo-ku, Kyoto, 606-8501, Japan}

\textit{Email}: \texttt{shinoda-mao@keio.jp}

\vspace{0.5cm}

\address{ Masaki Tsukamoto \endgraf
Department of Mathematics, Kyushu University, Moto-oka 744, Nishi-ku, Fukuoka 819-0395, Japan}

\textit{E-mail}: \texttt{masaki.tsukamoto@gmail.com}

\end{document}